\documentclass[a4paper,12pt]{article}
\usepackage[english]{babel}
\pdfoutput=1
\usepackage{amsfonts}
\usepackage{amsthm}
\usepackage{amsmath}
\usepackage{amssymb}
\usepackage{enumerate}
\usepackage{mathtools}
\usepackage{tikz-cd}
\usepackage[all]{xy}
\usepackage{graphicx}
\usepackage{caption}
\usepackage{amsmath}
\usepackage{extpfeil}
\usepackage{comment}
\usepackage{float}
\usepackage[toc]{appendix}
\usepackage[left=3.2cm,top=2.5cm,right=3.2cm,bottom=2.5cm]{geometry}
\usepackage{hyperref}
\usepackage{comment}
\usepackage{resizegather}
\usepackage[activate={true, nocompatibility}, final, tracking=true, kerning=true, spacing=true]{microtype}
\usepackage{tikz-cd}
\usepackage{subcaption}

\theoremstyle{plain}
\newtheorem{thm}{Theorem}[section]
\newtheorem*{thm*}{Theorem}
\newtheorem{coroll}[thm]{Corollary}
\newtheorem{defn}[thm]{Definition}
\newtheorem{lemma}[thm]{Lemma}

\newtheorem{prop}[thm]{Proposition}

\newtheorem{remark}[thm]{Remark}

\makeatletter
\newcommand{\colim@}[2]{%
  \vtop{\m@th\ialign{##\cr
    \hfil$#1\operator@font colim$\hfil\cr
    \noalign{\nointerlineskip\kern1.5\ex@}#2\cr
    \noalign{\nointerlineskip\kern-\ex@}\cr}}%
}
\newcommand{\colim}{%
  \mathop{\mathpalette\colim@{\rightarrowfill@\textstyle}}\nmlimits@
}
\makeatother

\tikzset{
  symbol/.style={
    draw=none,
    every to/.append style={
      edge node={node [sloped, allow upside down, auto=false]{$#1$}}}
  }
}

\microtypecontext{spacing=nonfrench}

\begin{document}
\thispagestyle{empty}

\title{Classification of affine normal $SL_2$-varieties with a dense orbit}
\author{Andres Fernandez Herrero and Rodrigo Horruitiner}
\date{}

\maketitle

\begin{abstract}
    In this note we give a classification of all affine normal $SL_2$-varieties containing an open dense orbit over an algebraically closed field of characteristic zero. Such a classification was first obtained by Popov in \cite{popov}. Here we provide an alternative approach.
\end{abstract}

\tableofcontents

\begin{section}{Introduction}

In this note we give a classification of affine normal $SL_2$-varieties with a dense open orbit over an algebraically closed field of characteristic zero. Our proof is based on elementary representation theory and the Hilbert-Mumford criterion.  

An affine algebraic normal $SL_2$-variety $X$ with a dense orbit is determined by a subalgebra $k[X]$ of the coordinate ring $k[SL_2]$ that is stable under the $SL_2$ action. This, in turn, is determined by an ``admissible'' subalgebra $k[X]^{\bar{U}}$ of $k[SL_2]^{\bar{U}}$ of invariants under the lower-triangular unipotent subgroup $\bar{U}$ of $SL_2$ (see Definition \ref{defn: admissible subalgebra}). 

The question is then to determine which subalgebras are admissible. In section 3 we study the multiplicative structure of the left regular representation $k[SL_2]$ and give a criterion for admissibility (Proposition \ref{prop: criterion for admissible subalgebras}). This criterion is a central ingredient in the proofs of the results, and it has a straightforward graphical interpretation inside the lattice of monomials in $k[SL_2]^{\bar{U}}$.

If $k[X]^{\bar{U}}$ has dimension 1, then $X$ is spherical (Proposition \ref{prop: general normal affine when dim 1}). In section 4 we use the Hilbert-Mumford criterion (Lemma \ref{lemma: hilbert mumford criterion}) and the admissibility criterion to obtain the classification in the case when $X$ is spherical (Theorem \ref{thm: classification of affine spherical vars of sl2}).

In section 5 we proceed to the case when $k[X]^{\bar{U}}$ has dimension 2. In this case we use the admissibility criterion to show that $X$ admits an extra torus action induced by the right regular representation on $k[SL_2]$ (Lemma \ref{lemma: general normal affine when dim 2 is right T stable}). This allows us to prove the main theorem of this paper:

\begin{thm*} \textbf{\ref{mainresult}}
Let $X$ be an affine normal $SL_2$-variety with an open dense orbit. Then $X$ is $SL_2$-isomorphic to exactly one of the following
\begin{enumerate}[(1)]
    \item A homogenenous space $SL_2/H$, where $H$ is as in Proposition \ref{prop: classification of homogeneous $SL_2$-vars}.
    \item (See Definition \ref{defn: Rf}) The spherical variety $SL_2 /\!/ (\mu_f \ltimes \overline{U}) = Spec\left(R_{\infty}^{(f)}\right)$ for a unique positive integer $f$.
    \item (See Definition \ref{defn: S_q^f}) $Spec\left(R^{(f)}_q\right)$ for a unique positive integer $f$ and a unique rational number $q \geq 1$ . In this case the stabilizer of the dense orbit is $\mu_f$.
\end{enumerate}
\end{thm*}

This classification is in fact complete in the case where $X$ has dimension $2$ without the assumption of normality. However, when $X$ has dimension $3$ it is not. See \cite[\S 4]{popov} for examples. 

The results presented here are known, and were first proved by Popov in $\cite{popov}$. Popov considers an equivariant linear embedding $X \subset V$ (\cite[Theorem 2]{popov}), and then proceeds to deduce intrinsic invariants of $X$ from invariants of the embedding (\cite[Theorem 4]{popov}). See also Kraft's book \cite[III.4]{kraft} for a shorter proof in the spirit of Popov's argument. Another approach to this classification can be found in the work of Luna and Vust  \cite[\S 9]{lunavust}, which uses their more general theory of spherical embeddings. 

Our approach, in contrast to the previous ones, is based on the combinatorics of the monomials of $k[X]^{\bar{U}}$ under the multiplicative structure of the left regular representation. This lends itself to a straightforward graphical interpretation of the results, and of the classifying invariants $(q,f)$ inside the lattice of monomials in $k[SL_2]^{\bar{U}}$. Even though the results are known, we hope that someone will find this new presentation to be useful.

See the work of Batyrev and Haddad \cite{batyrev} for applications of Theorem \ref{mainresult} above and further study of these varieties. In fact \cite{batyrev} gives an explicit categorical quotient description of the $3$-dimensional $SL_2$-varieties appearing in the classification. This is used to compute their Cox rings and to construct $SL_2$-equivariant flips as GIT quotients.

\textbf{Acknowledgements.} We would like to thank Dan Halpern-Leistner for asking the question that lead to the classification result, and Allen Knutson for making us aware of the Luna-Vust theory. We would also like to thank Victor Batyrev. from whom we learned about the results in \cite{batyrev}.

\end{section}

\begin{section}{Notation}
Let $k$ be an algebraically closed field of characteristic $0$. All schemes will be assumed to be defined over $k$. In this note we study varieties $X$ equipped with a left action of the algebraic group $SL_2$. We will abbreviate this by saying that $X$ is a $SL_2$-variety. 

We denote by $k[SL_2]$ the coordinate ring of the affine algebraic group $SL_2$. We think of a generic matrix in $SL_2$ as  having entries $\begin{bmatrix}
    a       & b  \\
    c       & d
\end{bmatrix}$. Then we have $k[SL_2] = k[a,b,c,d]/(ad-bc-1)$. The $k$-vector space $k[SL_2]$ acquires the structure of a $SL_2 \times SL_2$ representation. The first copy of $SL_2$ acts via the left regular representation, while the second copy of $SL_2$ acts via the right regular representation. We will employ the superscript $(-)^{op}$ whenever we are thinking of the second copy of $SL_2$. For example a $SL_2^{op}$-subrepresentation will be a subspace that is stable with respect to the right regular action.

We will denote by $\overline{B}$ the Borel subgroup in $SL_2$ consisting of lower triangular matrices. We let $T$ denote the maximal torus $T \subset \overline{B}$ defined by
\[ T = \left\{ \begin{bmatrix}
    t^{-1}       & 0  \\
    0       & t
\end{bmatrix}  \right\} \, \cong \mathbb{G}_m \]

We will write $\overline{U}$ for the unipotent radical of $\overline{B}$.
\end{section}

\begin{section}{Affine pointed $SL_2$-varieties with a dense orbit}
In this section we describe a correspondence between affine normal pointed $SL_2$-varieties with a dense orbit and certain class of subalgebras of the polynomial ring $k[a,b]$. We call this class of subalgebras ``admissible".

Let $X$ be an affine normal $SL_2$-variety. Let $R$ denote the coordinate ring of $X$. Choose $p \in X(k)$. We can define a map $\phi : SL_2 \rightarrow X$ given by $\phi(g) = g\cdot p$. Suppose further that the $SL_2$-orbit of $p$ is dense in $X$. This means that $\phi$ is dominant, which implies that we can view $R$ as a subalgebra of $k[SL_2]$. 

We can therefore understand such pointed $SL_2$-varieties once we determine the (left) $SL_2$-stable, normal, finitely generated subalgebras of $k[SL_2]$. By highest weight theory, these are determined by their $\overline{U}$-invariants.

Note that the $\overline{U}$-invariants of the left regular representation are $k[SL_2]^{\overline{U}} = k[a,b]$. Here the torus $T$ acts with weight $1$ on both $a$ and $b$. This motivates the following definition.
\begin{defn} \label{defn: admissible subalgebra}
Let $A$ be a subalgebra of $k[a,b]$. We say that $A$ is admissible if
\begin{enumerate}[(a)]
    \item $A$ is homogeneous for the grading by total degree (i.e. it is $T$-stable).
    \item $A$ is normal and finitely generated.
    \item Let $V_{A}$ denote the $SL_2$-subrepresentation of $k[SL_2]$ generated by $A$. Then $V_{A}$ is closed under multiplication (i.e. it is a subalgebra of $k[SL_2]$).
\end{enumerate}
\end{defn}

Now we state the promised correspondence.
\begin{lemma} \label{lemma: admissible subalgebras = pointed variaties with dense orbit}
 The map $A \mapsto V_A$ sets up a bijection between admissible subalgebras and $SL_2$-stable, normal, finitely generated $k$-subalgebras of the coordinate ring $k[SL_2]$.
\end{lemma}
\begin{proof}
Let's define an inverse map. For any $SL_2$-stable subalgebra $R$ of $k[SL_2]$, we define $A_R \vcentcolon = R^{\overline{U}} = R \cap k[a,b]$. This is a subalgebra of $k[a,b]$. It is $T$-stable because $R$ is $T$-stable. By highest weight theory, any representation of $SL_2$ is generated by its subspace of $\overline{U}$-invariants. So we get $V_{A_R} = R$. 

This establishes a bijection between $SL_2$-stable subalgebras of $k[SL_2]$ and subalgebras of $k[a,b]$ satisfying conditions (a) and (c) in Definition \ref{defn: admissible subalgebra}. In order to conclude the proof of the lemma, it suffices to show that $R$ is finitely generated and normal if and only if $A_{R}$ is finitely generated and normal. 

Suppose that $R$ is finitely generated. Then $R^{\overline{U}}$ is finitely generated by \cite{perri_intro_spherical}[Thm. 4.1.12]. Conversely, suppose that $A_{R}$ is finitely generated. Let $\{v_i\}_{i \in I}$ be a finite set of generators of $A_R$. Let $W_{\{v_i\}}$ denote the finite dimensional $SL_2$-submodule of $k[SL_2]$ generated by the $v_i$. Then we have that the algebra $k[W_{\{v_i\}}]$ generated by $W_{\{v_i\}}$ is a left $SL_2$-stable subalgebra of $R$. By assumption on the set $\{v_i\}$, we know that $k[W_{\{v_i\}}]$ contains $A_R$. Therefore, we must have $k[W_{\{v_i\}}] = R$ by highest weight theory. So $R$ is finitely generated.

The fact that $R$ is normal if and only if $A_{R}$ is normal follows from the argument in \cite{perri_intro_spherical}[Prop. 4.1.16 (ii)].
\end{proof}

The rest of this section is dedicated to proving Proposition \ref{prop: criterion for admissible subalgebras}. This is a criterion for determining when condition (c) in Definition \ref{defn: admissible subalgebra} is satisfied. 

We will need to understand the multiplicative structure of the left regular representation $k[SL_2]$. Consider the coordinate ring $W = k[a,b,c,d]$ of the vector space of $2 \times 2$ matrices. $SL_2$ acts on this ring by precomposing with the inverse of multiplication on the left. We have $W = \bigoplus_{i,j \in \mathbb{N}} W^{i,j}$ where
\[ W^{i,j} = Sym^i\left(ka \oplus kc\right) \otimes Sym^j\left(kb \oplus kd\right) \]

By representation theory of $SL_2$, we know that $W^{(i,j)}$ breaks up into a direct sum of irreducibles
\[ W^{i,j} = \bigoplus_{s = 0}^{min(i,j)} L^{i,j}_{i+j-2s}\]
Here $L^{i,j}_{i+j-2s}$ has highest weight $i+j-2s$ with respect to the Borel $\overline{B}$. It is apparent that $L^{i,j}_{i+j}$ has highest weight vector $a^ib^j$. In order to determine the highest weight vectors for the other summands, we will do a trick with the determinant. Note that multiplication by the invariant polynomial $ad-bc$ induces an injective map $W^{i,j} \hookrightarrow W^{i+1,j+1}$. By the decomposition above, this must induce isomorphisms $L^{i,j}_{i+j-2s} \xrightarrow{\sim} L^{i+1,j+1}_{i+j-2s}$. We conclude by using this map that $L^{i,j}_{i+j-2s}$ has highest weight vector $(ad-bc)^sa^{i-s}b^{j-s}$.

Thanks to the weight vectors obtained just above, we can describe explicitly the highest weight vectors of a tensor product of abstract irreducible representations of $SL_2$. But before we do that we need some notation.
\begin{defn} \label{defn: lowering operators}
Let $V$ be a representation of $SL_2$ and let $v \in V$ be nonzero weight vector with weight $n$. Then for every $e \in \mathbb{N}$, we define $\langle -e \rangle v$ to be the unique vectors determined by the equation
\[ \begin{bmatrix}
    1       & -x  \\
    0       & 1
\end{bmatrix} \cdot v = \sum_{e \in \mathbb{N}} \binom{n}{e}\,  x^e \,  \langle -e \rangle v \]

In other words, $\langle -d \rangle v$ is the component of weight $n-2e$ in the vector\\ $\begin{bmatrix}
    1       & -1  \\
    0       & 1
\end{bmatrix} \cdot \left(\binom{n}{e}^{-1} v\right)$
\end{defn}

\begin{prop} \label{prop: explicit highest weight vectors for tensors}
Let $V^1$, $V^2$ be two irreducible representations of $SL_2$ with highest weight $n_1$ and $n_2$ respectively. Let $v^1$ and $v^2$ be highest weight vectors of $V^1$ and $V^2$ respectively. The tensor product breaks up as a direct sum of irreducible representations
\[ V^1 \otimes V^2 = \bigoplus_{s =0}^{min(n_1,n_2)} V_{n_1+n_2-2s} \]
where $V_{n_1+n_2-2s}$ is an irreducible representation of highest weight $n_1 +n_2 -2s$ and highest weight vector
\[ v_{s} =  \sum_{e = 0}^s (-1)^e \binom{s}{e} \langle -e \rangle v_1 \otimes \langle e-s \rangle v_2  \]
\end{prop}
\begin{proof}
It suffices to prove this theorem for a specific realization of the irreducible representations $V^1$ and $V^2$. For example we can use $V^1 \cong Sym^{n_1}\left(ka \oplus kc\right)$ and $V^2 \cong Sym^{n_2}\left(kb \oplus kd\right)$. We can take $v^1 = a^{n_1}$ and $v^2 = b^{n_2}$. Now one can check that the vector $v_s$ described in this proposition reduces to $(ad-bc)^sa^{n_1-s}b^{n_2-s}$. We have seen in our discussion above that this is indeed the required highest weight vector.
\end{proof}

\begin{coroll} \label{coroll: explicit description of highest weight vectors for the product}
Let $p_1 = \sum_{i =0}^{n_1} z_{1,i} a^ib^{n_1-i}$ be a homogeneous polynomial of total degree $n_1$ in $k[a,b]$. Similarly, let $p_2 = \sum_{j =0}^{n_2} z_{2,j} a^jb^{n_2-j}$ be a homogeneous polynomial of degree $n_2$. Let $V^{p_1}$ (resp. $V^{p_2}$) be the irreducible subrepresentation of $W =k[a,b,c,d]$ with highest weight vector $p_1$ (resp. $p_2$). Then the tensor product decomposes into a direct sum
\[V^{p_1} \otimes V^{p_2} = \bigoplus_{s =0}^{min(n_2, n_2)} V^{p_1,p_2}_{n_1 + n_2 -2s}  \]
where $V^{p_1,p_2}_{n_1 + n_2 -2s}$ is an irreducible representation of highest weight $n_1+n_2-2s$. 

The image of a highest weight vector $v^{p_1,p_2}_{s}$ of $V^{p_1,p_2}_{n_1 + n_2 -2s}$ under the multiplication map $m: V^{p_1} \otimes V^{p_2} \rightarrow W$ is
\[ m(v^{p_1,p_2}_{s}) = \sum_{\alpha=s}^{n_1+n_2 -s} y_{\alpha,s}^{p_1,p_2} \, (ad-bc)^{s}\cdot a^{\alpha -s} \cdot b^{n_1+n_2-\alpha -s} \]
where $y_{\alpha, s}^{p_1,p_2}$ is given by
\[y_{\alpha,s}^{p_1,p_2} = \sum_{i+j = \alpha} \sum_{e=0}^s (-1)^e \frac{\binom{s}{e}}{\binom{n_1}{e} \binom{n_2}{s-e}} \binom{n_1-i}{e}\binom{n_2-j}{s-e} z_{1,i}z_{2,j}\]
\end{coroll}
\begin{proof}
We know that $m(v^{p_1,p_2}_{s})$ is a linear combination of highest weight vectors in $W$, because multiplication is $SL_2$-equivariant. We have seen that the highest weight vectors in $W$ of weight $n_1+n_2-2s$ and total degree $n_1+n_2$ in the variables $a,b,c,d$ are of the form $(ad-bc)^s\cdot a^{\alpha-s}\cdot b^{n_1+n_2-\alpha -s}$ for some $s \leq \alpha \leq n_1+n_2-s$. So we only need to determine the coefficients $y_{\alpha,s}^{p_1,p_2}$. This can be done by using the formula in Proposition \ref{prop: explicit highest weight vectors for tensors}, expanding with the binomial theorem and looking at the coefficient of $a^{\alpha} \cdot d^{s} \cdot b^{n_1+n_2-\alpha-s}$.
\end{proof}

We conclude this section with a useful criterion to determine if a subalgebra of $k[a,b]$ is admissible.
\begin{prop} \label{prop: criterion for admissible subalgebras}
Let $A \subset k[a,b]$ be a normal, finitely generated homogeneous subalgebra. Then $A$ is admissible if and only if for all pairs of homogeneous polynomials $p_1 = \sum_{i =0}^{n_1} z_{1,i} a^ib^{n_1-i}$ and $p_2 = \sum_{j =0}^{n_2} z_{2,j} a^jb^{n_2-j}$ in $A$, all of the polynomials $w^{p_1,p_2}_{s}$ defined for $0 \leq s \leq min(n_1,n_2)$ by
\[ w^{p_1,p_2}_{s} = \sum_{\alpha=s}^{n_1+n_2 -s} y_{\alpha,s}^{p_1,p_2} \, a^{\alpha -s}\, b^{n_1+n_2-\alpha-s}\]
with
\[y_{\alpha,s}^{p_1,p_2} = \sum_{i+j = \alpha} \sum_{e=0}^s (-1)^e \frac{\binom{s}{e}}{\binom{n_1}{e} \binom{n_2}{s-e}} \binom{n_1-i}{e}\binom{n_2-j}{s-e} z_{1,i}z_{2,j}\]
also belong to the subalgebra $A$.
\end{prop}
\begin{proof}
Let $V_A$ denote the $SL_2$ subrepresentation of $k[SL_2]$ generated by $A$. We can see that $V_A$ is closed under multiplication if and only if for all pairs of homogeneous polynomials $(p_1,p_2)$ we have that all of the highest weight vectors of the product $V^{p_1} \cdot V^{p_2}$ are also in $A$. By Proposition \ref{coroll: explicit description of highest weight vectors for the product} these highest weight vectors are just the polynomials $w^{p_1,p_2}_s$ (recall that $ad-bc =1$ in $k[SL_2]$).
\end{proof}

\begin{figure}[H]
\centering
\begin{tikzpicture}[scale=1.5]

\fill[gray!20] (0,0.5) -- (1.5,2) -- (0,3.5) -- (0,0.5);
\draw[step=.5cm,gray,very thin] (0,0) grid (2,3.6);
\draw [->] (0,0) -- (2.3,0) node[anchor=north]{$u$};
\draw [->] (0,0) --  (0,4) node[anchor=east]{$v$};
\draw [-,thick] (0,1.5) -- (0.5,1) node[fill=none,anchor=west]{$p_1$};

\draw [-,thick] (0,2) -- (1,1)    node[fill=none,anchor=west]{$p_2$};

\draw [-,] (1.5,2) -- (0,3.5)  node[fill=none,anchor=east]{$p_1 p_2 = w_0^{p_1, p_2}$};
\draw [-,] (1,1.5) -- (0,2.5)  node[fill=none,anchor=east]{$w_1^{p_1, p_2}$};
\draw [-,] (0.5,1) -- (0,1.5)  node[fill=none,anchor=east]{$w_2^{p_1, p_2}$};
\draw [-,] (0,0.5) -- (0,0.5)  node[fill=none,anchor=east]{$w_3^{p_1, p_2}$};

\end{tikzpicture}
\caption{Admissibility criterion for two homogeneous polynomials $p_1 = ab^2 + b^3$, $p_2 = a^2 b^2 + ab^3 + b^4$}\label{diagram example}.
\end{figure}
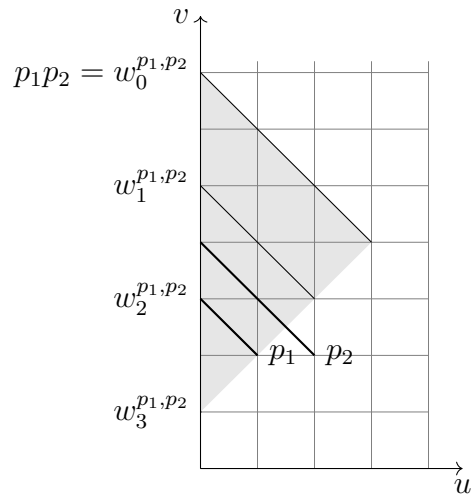

This criterion has a simple graphical interpretation. Give two polynomials $p_1$ and $p_2$ in $A$, we may consider their Newton polygons (i.e. the convex hull of points $(u,v) \in \mathbb{N}^2$ corresponding to monomials $a^u b^v$ in their support) and draw a diagram as in Figure \ref{diagram example} above.

As long as the coefficients $y_{\alpha,s}^{p_1,p_2}$ are nonzero, we can deduce that some elements of convenient form must be present in the algebra $A$ starting from some known elements $p_1$, $p_2$ (for example, $w_3^{p_1,p_2} = y^{p_1, p_2}_{3,3} b$ in the image).

\end{section}

\begin{section}{Classification of affine spherical $SL_2$-varieties}
For this section, we will put an additional assumption on our affine pointed $SL_2$-variety $(X, p)$. Recall that we can use the point $p$ to view $X$ as $Spec(R)$ for some $SL_2$-stable subalgebra $R \subset k[SL_2]$.
\begin{lemma} \label{lemma: criterion spherical multiplicites}
Let $X = Spec(R)$ be an affine normal pointed variety as above. Then the following statements are equivalent.
\begin{enumerate}[(1)]
\item The orbit $\overline{B}\cdot p$ is dense.
\item Every $T$-weight space of the algebra of invariants $R^{\overline{U}} = R \cap k[a,b]$ is one-dimensional.
\end{enumerate}
\end{lemma}
\begin{proof}
See \cite{perri_intro_spherical}[Lemma 4.2.2].
\end{proof}

As a consequence we get the following proposition.
\begin{prop} \label{prop: admisible prop in spehrical case}
The correspondence $A \mapsto V_A$ defines a bijection between the following two sets
\begin{enumerate}[(a)]
    \item Admissible subalgebras $A$ such that for all $n \in \mathbb{N}$ the $k$-space of homogeneous polynomials of degree $n$ in $A$ has dimension at most $1$.
    \item Affine normal pointed $SL_2$-varieties $X = Spec(V_A)$ such that $\overline{B}\cdot p$ is dense.
\end{enumerate}
\end{prop}
\begin{proof}
This is an immediate consequence of Lemma \ref{lemma: admissible subalgebras = pointed variaties with dense orbit} and Lemma \ref{lemma: criterion spherical multiplicites}.
\end{proof}

\begin{defn}
An affine normal $SL_2$-variety $X$ is called spherical if it has a $\overline{B}$-orbit that is dense.
\end{defn}

In order to classify spherical varieties, we will need a couple of lemmas. Recall that a one-parameter subgroup of $SL_2$ is a homomorphism of algebraic groups $\gamma: \mathbb{G}_m \rightarrow SL_2$. Let $\mathbb{G}_m \hookrightarrow \mathbb{A}^1$ be the usual open immersion. Given a morphism $\phi: \mathbb{G}_m \rightarrow X$ we say that $\text{lim}_{t \to 0} \, \phi(t)$ exists if the map $\phi$ extends to a morphism $\widetilde{\phi}: \mathbb{A}^1 \rightarrow X$. If this is the case, we will use the notation $\text{lim}_{t \to 0} \, \phi(t) = \widetilde{\phi}(0)$.

\begin{lemma} \label{lemma: hilbert mumford criterion}
Let $X$ be an affine $SL_2$-variety and let $p \in X(k)$. Suppose that the orbit $SL_2 \cdot p$ is open dense in $X$. For any closed point $x \in X(k) \setminus (SL_2(k) \cdot p)$, there exists a nonconstant one-parameter subgroup $\gamma : \mathbb{G}_m \rightarrow SL_2$ such that $\text{lim}_{t \to 0} \, \gamma(t) \cdot p$ is contained in the closure of the orbit $SL_2 \cdot x$.
\end{lemma}
\begin{proof}
This is a version of the Hilbert-Mumford criterion. It follows from a modification of the argument in \cite{mumford_git}[pg.53-54] after embedding $X$ into the total space of a linear representation of $SL_2$.
\end{proof}

Let $\gamma: \mathbb{G}_m \rightarrow SL_2$ be a nontrivial one-parameter subgroup. Recall that there is a Borel subgroup $B_{\gamma}$ associated to $\gamma$. It is defined by
\[ B_{\gamma} = \left\{ \, g \in SL_2 \, \, \mid \, \, \text{lim}_{t \to 0} \, \gamma(t)\,g \, \gamma(t)^{-1} \; \; \text{exists in $SL_2$}\, \right\} \]
The image $\text{Im} \, \gamma$ is a maximal torus inside $B_{\gamma}$. The dominant weights of $B_{\gamma}$ are the characters of $\text{Im} \, \gamma$ that pullback to a nonnegative weight of $\mathbb{G}_m$ under the homomorphism $\gamma$.

For the next lemma, we let $(X, p)$ be an affine pointed variety such that $SL_2 \cdot p$ is dense. As usual, we view the coordinate ring $R$ as a $SL_2$-stable subalgebra of $k[SL_2]$. Recall that we use the supercript $(-)^{op}$ to denote the right regular action.
\begin{lemma} \label{lemma: criterion for spherical classes of limits}
 Let $\gamma : \mathbb{G}_m \rightarrow SL_2$ be a non-constant one-parameter subgroup. Let $B_{\gamma}$ denote the corresponding Borel subgroup. Then the following statements are equivalent
 \begin{enumerate}[(1)]
 \item The limit $\text{lim}_{t \to 0}  \, \gamma(t) \cdot p$ exists in $X$
 \item The subalgebra $R \subset k[SL_2]$ is contained in the sum of those $T^{op}$-weight spaces of $k[SL_2]$ that are dominant for the Borel $B_{\gamma}^{op}$.
 \end{enumerate}
\end{lemma}
\begin{proof}
Consider the composition
\begin{align*}
\varphi: \; \; SL_2 \times \mathbb{G}_m  \; \;\xrightarrow{m} \; \; SL_2  \; \;\longrightarrow \; \; X \\
(g,t) \; \; \; \; \mapsto \; \; \; gt \; \; \mapsto \; \; \; \; g \gamma(t) \cdot p
\end{align*}

At the level of coordinate rings, we have
\[\varphi^* : R \; \hookrightarrow \; k[SL_2] \; \xrightarrow{m^*} \; k[SL_2]\otimes k[T, T^{-1}]\]

Notice that $\mathbb{G}_m$ acts by multiplication on the right on $k[SL_2]$ via the homomorphism $\gamma$. The corresponding comodule map is $m^*$ above. If $f \in R$, we can break $f$ into a sum of weight vectors $f = \sum_{n \in \mathbb{Z}} f_{n}$ with $f_{n} \in k[Sl_2]$ having weight $n$. By definition we have
\[ \varphi^*(f) = \sum_{n \in \mathbb{Z}} f_{n} \otimes T^{n} \]

Assume that $\text{lim}_{t \to 0}\, \gamma(t) \cdot p$ exists. By definition, this means that the map $\varphi|_{\{\text{id}\}\times \mathbb{G}_m} : \mathbb{G}_m \rightarrow X$ extends to a map $\widetilde{\varphi}_0 : \mathbb{A}^1 \rightarrow X$. Therefore, we get a $G$-equivariant map
\begin{align*}
    \widetilde{\varphi} :\; \; G \times \mathbb{A}^1 \; \; \; \longrightarrow \; \; X \\
    (g,t) \; \; \; \mapsto \; \; \;  g \cdot \widetilde{\varphi}_0(t)
\end{align*}

By construction this map agrees with $\varphi$ on $G \times \mathbb{G}_m$. We conclude that the map $\varphi^*: R \rightarrow k[SL_2] \otimes k[T,T^{-1}]$ factors through the subring $k[SL_2] \otimes k[T] \subset k[SL_2] \otimes k[T,T^{-1}]$. This means that for all $f \in R$ we have $f_n = 0$ for $n <0$. In other words, $f$ belongs to the sum of weight spaces of the right regular representation $k[SL_2]$ that are dominant for $B_{\gamma}$.

Conversely, suppose that for any $f \in R$ we have $f_n = 0$ for all $n<0$. This implies that the morphism $\varphi: G \times \mathbb{G}_m \rightarrow X$ extends to to a morphism $\widetilde{\varphi}: G \times \mathbb{A}^1 \rightarrow X$. We can restrict $\widetilde{\varphi}$ to $\{id\} \times \mathbb{A}^1$ in order to deduce that the limit $\text{lim}_{t \to 0} \,  \gamma(t) \cdot p$ exists.
\end{proof}

The following is the main step in classifying affine spherical $SL_2$-varieties.
\begin{prop} \label{prop: classification of spherical varieties cases}
Let $X$ be an affine spherical $SL_2$-variety. Then $X$ is $SL_2$-isomorphic to one of the following
\begin{enumerate}[(a)]
    \item A homogeneous space $G/H$ for some closed algebraic subgroup $H \subset G$.
    
    \item $Spec\left(V_{k[b^f]}\right)$ for some positive integer $f$.
    
    \item $Spec\left(V_{k[ab]}\right)$.
    
    \item $Spec\left(V_{k[(ab)^2]}\right)$.
\end{enumerate}
\end{prop}
\begin{proof}
Choose a point $p$ in the dense $\overline{B}$-orbit. This allows us to view $X$ as $Spec(R)$ for some $SL_2$-stable subalgebra $R \subset k[SL_2]$.

Assume that $X$ is not a homogeneous space. By the Hilbert-Mumford criterion (Lemma \ref{lemma: hilbert mumford criterion}) there exists a nontrivial one-parameter subgroup $\gamma: \mathbb{G}_m \rightarrow SL_2$ such that $\text{lim}_{t \to 0} \, \gamma(t)\cdot p$ exists. Let $B_{\gamma}$ be the corresponding Borel subgroup. Notice that right multiplying the algebra $R$ by an element $g \in SL_2(k)$ establishes an $SL_2$-isomorphism $Spec(R) \xrightarrow{\sim} Spec(R \cdot g)$. By Lemma \ref{lemma: criterion for spherical classes of limits}, we have that $\text{lim}_{t \to 0} \, g^{-1} \gamma(t) g\cdot p$ exists in $Spec(R\cdot g)$. The corresponding Borel subgroup will be $g^{-1} B_{\gamma} g$. By choosing an appropriate element $g \in SL_2(k)$ we can assume without loss of generality that $B_{\gamma} = \overline{B}$.

 Let $A = R \cap k[a,b]$ be the associated admissible subalgebra of $k[a,b]$. If $A$ is the field of constants $k$ then we have $R =k$. So the corresponding variety $X = Spec(k)$ is the homogeneous space $SL_2/SL_2$, which is not the case by assumption. Therefore $A$ contains some homogeneous polynomial with positive degree in $k[a,b]$. 
 
 Let $p$ denote a homogeneous polynomial of smallest positive degree $n$ in $A$. We claim that $A = k[p]$. Indeed let $S$ denote the set of homogeneous polynomials in $A$. It follows from the multiplicity condition $(2)$ in Lemma \ref{lemma: criterion spherical multiplicites} that all the weight spaces of the localization $A[S^{-1}]$ have dimension $1$. Let $l\mathbb{Z}\subset \mathbb{Z}$ denote the subgroup of weights of $A[S^{-1}]$, where $l$ is a positive integer. If $l=n$ then multiplicity one implies that $A[S^{-1}] = k[p^{\pm 1}]$. Therefore $A = k[p]$. If $l\neq n$ then $l$ is a proper divisor of $n$. Let $p'$ be an element of $A[S^{-1}]$ of weight $l$. By multiplicity one for the weight spaces, we must have $(p')^{\frac{n}{l}} \in k \,p$. Since $A$ is normal, this implies $p' \in A$. So we get a contradiction to the minimality of the degree $n$. This concludes the proof of the claim $A = k[p]$.
    
Suppose that $p = \sum_{i =0}^n z_i a^ib^{n-i}$. By Lemma \ref{lemma: criterion for spherical classes of limits}, $p$ must be contained in a sum of $\overline{B}^{op}$-dominant weight spaces. Concretely, this means that $z_i = 0$ whenever $n-i < i$. Let us denote by $m$ the largest index such that $z_m \neq 0$. We must have $m \leq \lfloor \frac{n}{2} \rfloor$.
    
Let's apply the criterion in Proposition \ref{prop: criterion for admissible subalgebras} to the pair $(p,p)$. For each $0 \leq s \leq n$, we want to determine whether the polymomial $w^{p,p}_{s}$ is in the subalgebra $A = k[p]$. Recall that
    \[ w^{p,p}_{s} = \sum_{\alpha=s}^{2n -s} y_{\alpha,s}^{p,p} \, a^{\alpha -s}\cdot b^{n_1+n_2-\alpha-s}\]
    with 
    \[y_{\alpha,s}^{p,p} = \sum_{i+j = \alpha} \sum_{e=0}^s (-1)^e \frac{\binom{s}{e}}{\binom{n}{e} \binom{n}{s-e}} \binom{n-i}{e}\binom{n-j}{s-e} z_iz_j\]
    
    Note that $y_{\alpha, s}^{p,p} = 0$ for all $\alpha > 2m$. This is because $z_i= 0$ whenever $i >m$. This means that for $s =2m$ we have a single term
    \[ w^{p,p}_{2m} = y_{2m,2m}^{p,p} \, b^{2n-4m}\]
    with
    \[y_{2m,2m}^{p,p} = 2 z_m^2\cdot \sum_{e=0}^{2m} (-1)^e \frac{\binom{2m}{e}}{\binom{n}{e} \binom{n}{2m-e}} \binom{n-m}{e}\binom{n-m}{2m-e}\]
We have $z_m \neq 0$ by assumption. Lemma \ref{lemma: first binomial sum lemma} (with $i = m$ and $j = n-m$) implies that $y_{2m,2m}^{p,p} \neq 0$. This means that $b^{2n-4m}$ must be a constant multiple of a power of $p$. We have four possible cases.
    \begin{enumerate}[(C1)]
        \item $n \neq 2m$. Then $b^{2n-4m}$ is not a scalar. This means that $p$ can be taken to be a power of $b$. So $A = k[b^f]$ for some positive integer $f$. Therefore $X$ satisfies $(b)$ in the proposition.
        
        \item $n = 2$ and $m=1$. Then up to multiplying $p$ by a scalar, we can write $p = ab + zb^2$ for some $z \in k$. Right multiplication by the element $\begin{bmatrix} 1 & 0 \\ -z & 1 \end{bmatrix}$ establishes a $SL_2$-isomorphism between $k[ab-zb^2]$ and the algebra $k[ab]$. This means that $X$ satisfies $(c)$ in the proposition.
        
        \item $n=4$ and $m=2$. Then $p$ can be chosen to be of the form $p = a^2b^2 +z_1ab^3 + z_0b^4$. We can right multiply by $\begin{bmatrix}1&0 \\-\frac{1}{2}z_1 & 1 \end{bmatrix}$ in order reduce to $z_1 = 0$. So $A = k[p] = k[a^2b^2+z_0b^4]$. We claim that this algebra is admissible only if $z_0 =0$. Indeed, the vector $w^{p,p}_2$ has a single nonzero summand in this case
        \[ w^{p,p}_2= y_{4,2}^{p,p} a^2b^2\]
        with
        \begin{gather*}
            y_{2,2}^{p,p} = \sum_{e=0}^2 (-1)^e \frac{\binom{2}{e}}{\binom{4}{e} \binom{4}{2-e}}\binom{2}{e}\binom{4}{2-e} z_0 + \sum_{e=0}^2 (-1)^e \frac{\binom{2}{e}}{\binom{4}{e} \binom{4}{2-e}}\binom{4}{e}\binom{2}{2-e} z_0
        \end{gather*}
        This simplifies to
        \[y_{2,2}^{p,p} = \frac{1}{3}z_0\]
        So we have $w^{p,p}_2 = \frac{1}{3}z_0a^2b^2$. This can't be in $A = k[p]$ unless $z_0 = 0$.
        
        We conclude that $A = k[(ab)^2]$. Therefore $X$ satisfies $(d)$ in the proposition.
        
         \item $n = 2m$ and $m > 2$. We will show that this case can't actually happen.
         
         We can write $p = z_m(ab)^m + \sum_{i =0}^{m-1}z_i a^i b^{2m-i}$ with $z_m \neq 0$. Let's look at the highest weight vector $w^{p,p}_2$. Notice that the degree of $w^{p,p}_2$ is $4m-4$. Since $m>2$, the only way $w^{p,p}_2$ can be a multiple of a power of $p$ is if $w^{p,p}_2 =0$ (by degree count). This means in particular that $y_{2m,2}^{p,p}$ must be $0$. Since $z_{i} = 0$ for $i>m$, the expression for $y_{2m,2}^{p,p}$ simplifies to
         \begin{gather*}
            y_{2m,2}^{p,p} = \sum_{e=0}^2 (-1)^e \frac{\binom{2}{e}}{\binom{2m}{e} \binom{2m}{2-e}} \binom{m}{e}\binom{m}{2-e} z_m^2
        \end{gather*}
         By Lemma \ref{lemma: second binomial sum lemma} with $i=0$, we have
        \[y_{2m,2}^{p,p} = z_{m}^2 \cdot \frac{-1}{2m-1}\]
        So we see that $y_{2m,2}^{p,p} = 0$ implies that $z_{m} = 0$, a contradiction.
        \end{enumerate}

\end{proof}

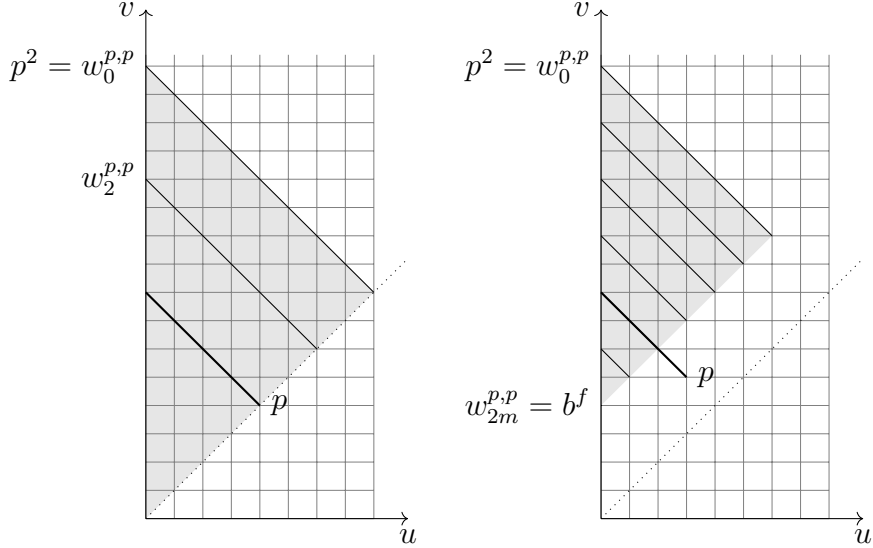
\begin{figure}[ht!]
\centering
\begin{subfigure}[b]{0.4\textwidth}

\begin{tikzpicture}[scale=1.5]

\fill[gray!20] (0,0) -- (2,2) -- (0,4) -- (0,0);
\draw[step=.25cm,gray,very thin] (0,0) grid (2,4.1);
\draw [->] (0,0) -- (2.3,0) node[anchor=north]{$u$};
\draw [->] (0,0) --  (0,4.5) node[anchor=east]{$v$};
\draw[dotted] (0,0) -- (2.3,2.3);

\draw [-,thick] (0,2) -- (1,1)    node[fill=none,anchor=west]{$p$};

\draw [-,] (2,2) -- (0,4)  node[fill=none,anchor=east]{$p^2 = w_0^{p,p}$};

\draw [-,] (1.5,1.5) -- (0,3)  node[fill=none,anchor=east]{$w_2^{p, p}$};

\end{tikzpicture}
\end{subfigure}
\begin{subfigure}[b]{0.4\textwidth}
\begin{tikzpicture}[scale=1.5]

\fill[gray!20] (0,1) -- (1.5,2.5) -- (0,4) -- (0,1);
\draw[step=.25cm,gray,very thin] (0,0) grid (2,4.1);
\draw [->] (0,0) -- (2.3,0) node[anchor=north]{$u$};
\draw [->] (0,0) --  (0,4.5) node[anchor=east]{$v$};
\draw[dotted] (0,0) -- (2.3,2.3);

\draw [-,thick] (0,2) -- (0.75,1.25)    node[fill=none,anchor=west]{$p$};

\draw [-,] (1.5,2.5) -- (0,4)  node[fill=none,anchor=east]{$p^2 = w_0^{p,p}$};
\draw [-,] (1.25,2.25) -- (0,3.5)  node[fill=none,anchor=east]{};
\draw [-,] (1,2) -- (0,3)  node[fill=none,anchor=east]{};
\draw [-,] (0.75,1.75) -- (0,2.5)  node[fill=none,anchor=east]{};
\draw [-,] (0.5,1.5) -- (0,2)  node[fill=none,anchor=east]{};
\draw [-,] (0.25,1.25) -- (0,1.5)  node[fill=none,anchor=east]{};
\draw [-,] (0,1) -- (0,1)  node[fill=none,anchor=east]{$w_{2m}^{p,p} = b^f$};

\end{tikzpicture}
\end{subfigure}
\caption{Application of the criterion in Proposition \ref{prop: criterion for admissible subalgebras} to $(p,p)$ in the proof of Proposition \ref{prop: classification of spherical varieties cases}. By Lemma \ref{lemma: criterion for spherical classes of limits}, $p$ lies above the dotted line as shown. Cases (C4) on the left and (C1) on the right.}
\end{figure}

The following proposition shows that cases $(b)$, $(c)$ and $(d)$ can actually arise.
\begin{prop} \label{prop: admisibility of spherical algebras}
For any positive integer $f$, the algebra $k[b^f]$ is admissible. The algebras $k[(ab)]$ and $k(ab)^2]$ are also admissible.
\end{prop}
\begin{proof}
Set $R_f = k[b^f, db^{f-1}, d^2b^{f-2}, ..., d^f]$ inside $k[SL_2]$. Note that $R_f$ is $SL_2$-stable. We have $k[b^f] = R_f \cap k[a,b]$. This shows that $V_{k[b^f]} = R_f$ in part $(b)$ of the previous proposition.

It can be checked that we have $k[ab] = k[ab,ad,bc,cd] \, \cap \, k[a,b]$. So in part $(c)$ we have $V_{k[ab]} =  k[ab,ad,bc,cd]$.

Finally, one can check in a similar way that $V_{k[(ab)^2]}$ is the subalgebra\\ $k[a^2b^2, ab(ad+bc), cd(ad+bc), c^2d^2]$ inside $k[SL_2]$.
\end{proof}

\begin{defn} \label{defn: Rf}
Let $f$ be a positive integer. We define $R_{\infty}^{(f)}$ to be the $SL_2$-stable subalgebra of $k[SL_2]$ generated by $k[b^f]$. More concretely, 
\[R_{\infty}^{(f)} = k[b^f, db^{f-1}, d^2b^{f-2}, ..., d^f]\]
Note that $R_{\infty}^{(f)}$ is the subalgebra of $(\mu_f \ltimes \overline{U})^{op}$-invariants $k[SL_2]^{(\mu_f \ltimes \overline{U})^{op}}$. Here $\mu_f \ltimes \overline{U} \subset T \ltimes \overline{U} = \overline{B}$.
\end{defn}

Now it remains to classify the affine spherical homogeneous spaces.
\begin{prop} \label{prop: classification of spherical homogeneous spaces}
Up to isomorphism, the affine spherical homogeneous spaces for $SL_2$ are
\begin{enumerate}[(a)]
    \item $SL_2/SL_2 = Spec(k)$.
    \item $SL_2/T$. This is isomorphic to $Spec\left(V_{k[ab]}\right)$.
    \item $SL_2/N_T$. This isomorphic to $Spec\left(V_{k[(ab)^2}\right)$.
\end{enumerate}
\end{prop}
\begin{proof}
By Matsushima's criterion \cite{richardson-affinehomegeneous} the homogeneous space $SL_2/H$ is affine if and only if $H$ is reductive. If $SL_2/H$ has a dense $\overline{B}$-orbit, then by dimension count we must have $\text{dim}\, H \geq 1$. By the classification of connected reductive groups, we must have $\text{dim} H = 1 \text{ or } 3$. If $\text{dim}\, H= 3$, then $H = SL_2$. If $\text{dim} \, H =1$, then the neutral component of $H$ must be a torus. By conjugating we can assume that the neutral component of $H$ is $T$. So we must have either $H = T$ or $H$ is the normalizer of the torus $N_T= \mathbb{Z}/2\mathbb{Z} \ltimes T$.

The explicit description of the algebras follows from taking the subalgebra of $H^{op}$-invariants of $k[a,b]$ in each case.
\end{proof}

\begin{thm} \label{thm: classification of affine spherical vars of sl2}
Every affine spherical $SL_2$-variety is $SL_2$-isomorphic to exactly one of the following
\begin{enumerate}[(a)]
    \item $SL_2/SL_2 = Spec(k)$.
    \item $SL_2/T$.
    \item $SL_2/N_T$.
    \item (See Definition \ref{defn: Rf}) \; $SL_2 /\!/ (\mu_f \ltimes \overline{U}) = Spec\left(R_{\infty}^{(f)}\right)$ for a unique positive integer $f$.
\end{enumerate}
\end{thm}
\begin{proof}
The fact that this list is exhaustive follows from Proposition \ref{prop: classification of spherical varieties cases} and Proposition \ref{prop: classification of spherical homogeneous spaces}. We just need to show that no two of these spherical varieties are $SL_2$-isomorphic. $SL_2/SL_2$ is the only one that has dimension $0$, so it can't be isomorphic to the others.

In order to study the other ones, we can look at the algebras of $\overline{U}$-invariants. Recall by Proposition \ref{prop: admisibility of spherical algebras}
\begin{enumerate}[]
    \item $k[SL_2/T]^{\overline{U}} \cong k[ab]$.
    \item $k[SL_2/N_T]^{\overline{U}} \cong k[(ab)^2]$.
    \item $\left(R_{\infty}^{(f)}\right)^{\overline{U}} \cong k[b^f]$.
\end{enumerate}

These are isomorphisms as algebras with a left $T$-action. Note that most of these algebras are not isomorphic even as $T$-representations. In fact, the only ones that are isomorphic as $T$-representations are $k[b^2] \cong k[ab]$ and $k[b^4] \cong k[(ab)^2]$. We are therefore reduced to showing that $SL_2/T \ncong Spec\left(R_{\infty}^{(2)}\right)$ and $SL_2/N_T \ncong Spec\left(R_{\infty}^{(4)}\right)$ as $SL_2$-varieties.

$R_{\infty}^{(2)} = k[b^2,bd,d^2]$ is isomorphic as a ring to $k[x,y,z]/(y^2-xz)$. Note that this ring has a singularity at the origin $x=y=z=0$. So $Spec\left(R_{\infty}^{(2)}\right)$ is not smooth, and therefore it can't be isomorphic to $SL_2/T$.

A similar reasoning shows that $Spec\left(R_{\infty}^{(4)}\right)$ is not smooth, and so it can't possibly be isomorphic to $SL_2/N_T$.
\end{proof}

\begin{remark}
$Spec\left(R_{\infty}^{(f)}\right)$ will have a singular point for $f>1$. Indeed, $Spec\left(R_{\infty}^{(f)}\right)$ is isomorphic to the affine cone over the degree $f$ Veronese embedding of $\mathbb{P}^1$ into $\mathbb{P}^{f}$. For $f>1$, this will be an affine normal surface with a singular point (the cone point).
\end{remark}
\end{section}
\begin{section}{Classification of affine normal $SL_2$-varieties with an open dense orbit}
In this section we will apply similar techniques to classify all $SL_2$-isomorphism classes of affine normal $SL_2$-varieties with an open dense orbit. We start by dealing with the homogeneous ones.
\begin{prop} \label{prop: classification of homogeneous $SL_2$-vars}
Let $X = SL_2/H$ be an affine homogeneous $SL_2$-variety. Then up to $SL_2$-isomorphism, exactly one of the following holds
\begin{enumerate}[(a)]
    \item $X$ is spherical. We have either $H = T$, $H=N_T$ or $H = SL_2$.
    \item $(Type A)$ $H$ is the $f$-torsion subgroup $\mu_f \subset T$ for some positive integer $f$.
    \item $(Type D)$ $H$ is the subgroup $ \mathbb{Z}/2\mathbb{Z} \ltimes \mu_f$ inside $N_T$.
    \item $(E_6)$ $H \subset SL_2$ is the binary tetrahedral group $ \mathbb{T}$.
    \item $(E_7)$ $H \subset SL_2$ is the binary octahedral group $\mathbb{O}$.
    \item $(E_8)$ $H \subset SL_2$ is the binary icosahedral group $\mathbb{I}$.
\end{enumerate}
\end{prop}
\begin{proof}
If $\text{dim} \, H \geq 1$ then we are in case $(a)$ by the proof of Proposition \ref{prop: classification of spherical homogeneous spaces}. If $\text{dim}\, H = 0$, then $H$ is a finite subgroup of $SL_2$. We are reduced to classifying finite subgroups of $SL_2$ up to conjugation. Such classification is well known, see \cite{springer_invariant_thry}[4.4] for a treatment. The classification includes exactly the cases $(b)$ through $(f)$ above.
\end{proof}

For completeness, let's recall explicit descriptions of the exceptional cases $(d)$, $(e)$ and $(f)$ above. See \cite{slodowy_simple_sing}[\S 6.1] or \cite{springer_invariant_thry}[4.4] for details.
\begin{enumerate}
    \item[$(E_6)$] $\mathbb{T}$ is generated by the group $\mathbb{Z}/2\mathbb{Z} \ltimes \mu_2 \subset N_T$ and the matrix $\frac{1}{\sqrt{2}} \begin{bmatrix} \epsilon^7 & \epsilon^7 \\ \epsilon^5 & \epsilon \end{bmatrix}$. Here $\epsilon$ is a primitive $8$-th root of unity.
    \item[$(E_7)$] $\mathbb{O}$ is generated by $\mathbb{T}$ and the matrix $\begin{bmatrix} \epsilon & 0 \\ 0 & \epsilon^7 \end{bmatrix}$, where $\epsilon$ is the same $8$th root of unity as above.
    \item[$(E_8)$] $\mathbb{I}$ is generated by the matrices \;$- \begin{bmatrix} \eta^3 & 0 \\ 0 & \eta^2 \end{bmatrix}$ and \; $\frac{1}{\eta^2 - \eta^3} \begin{bmatrix} \eta + \eta^4 & 1 \\ 1 & -\eta -\eta^4 \end{bmatrix}$, where $\eta$ is a primitive $5$th root of unity.
\end{enumerate}

The admissible subalgebra $A \subset k[a,b]$ in each homogeneous case $SL/H$ can be computed as the ring of right invariants $k[a,b]^{H^{op}}$ (we take the distinguished point $p$ to be the identity). In Proposition \ref{prop: classification of spherical homogeneous spaces} we have already described $A$ when $\text{dim} \, H \geq  1$. See \cite{springer_invariant_thry}[Ch. 4] or \cite{dolgachev_mackay}[Ch. 1] for details on the computation of these algebras when $H$ is a finite group.

It remains to classify the nonhomogeneous $SL_2$-varieties. We will split the classification into two cases, depending on the dimension of $k[X]^{\overline{U}}$.

\begin{prop} \label{prop: general normal affine when dim 1}
Let $X$ be an affine normal $SL_2$-variety with an open dense orbit. Assume that $X$ is not a homogeneous $SL_2$-variety. Suppose that the ring $k[X]^{\overline{U}}$ has dimension 1. Then $X$ is spherical. In particular $X$ is as in part $(d)$ of Theorem \ref{thm: classification of affine spherical vars of sl2}.
\end{prop}
\begin{proof}
Choose a closed point $p$ in the dense open orbit of $X$. This allows us to view $X$ as $Spec(R)$ for some left $SL_2$-subalgebra $R \subset k[SL_2]$. Since $X$ is not a homogeneous $SL_2$-variety, there is some closed point in the complement of the open dense orbit. By the Hilbert-Mumford criterion \ref{lemma: hilbert mumford criterion}, we know that there exists a nonconstant one-parameter subgroup $\gamma$ such that $\text{lim}_{t \to 0} \, \gamma(t) \cdot p$ exists in $Spec(R)$. After right multiplying by and element of $SL_2(k)$, we can assume that $B_{\gamma} = \overline{B}$.

Let $A = R^{\overline{U}}= R\cap k[a,b]$ be the corresponding admissible algebra. By assumption, $\text{dim}\, A = 1$. Therefore $A \neq k$. This means that there exist some homogeneous polynomial $p \in A$ with positive degree. If $k[p] = A$, then $X$ is spherical by Proposition \ref{prop: admisible prop in spehrical case}.

Assume that $A \neq k[p]$. We claim that $b^f \in A$ for some $f \geq 1$. Up to a scalar multiple, we can write $p = a^m b^{n-m} +\sum_{i=0}^{m-1}z_ia^ib^{n-i}$ for some $m \leq n$. Since $B_{\gamma} = \overline{B}$, we conclude by Lemma \ref{lemma: criterion for spherical classes of limits} that $m \leq \left\lfloor \frac{n}{2} \right\rfloor$. If $m \neq \frac{n}{2}$, then the argument in Proposition \ref{prop: classification of spherical varieties cases} Case (C1) shows that the element $w^{p,p}_{2m}$ is of the form $zb^f$ for some $z \in k \setminus \{0\}$ and $f \geq 1$. So we would be done with the claim in this case. Suppose therefore that $m = \frac{n}{2}$.

Since $k[p] \neq A$, there is a nonzero homogeneous polynomial $h \in A \setminus k[p]$. We can write $h = a^sb^{t-s} + \sum_{i=0}^{s-1} z'_i a^i b^{t-i}$. If $s \neq \frac{t}{2}$, then by the same reasoning as above we can conclude that $b^f \in A$ for some $f \geq 1$. Otherwise, we can look at the nonzero polynomial $p^t - h^n$. Notice that this is a homogeneous polynomial of degree $nt$. The highest power of the variable $a$ must be strictly smaller than $\frac{nt}{2}$ because we are cancelling the monomials with the highest $a$-exponents in $p^t$ and $h^m$. Therefore, we can just apply the argument of \ref{prop: classification of spherical varieties cases}, Case (C1) to deduce that $b^f \in A$ for some $f \geq 1$.

Let $f$ denote the minimal positive integer such that $b^f \in A$ (i.e. $k[b^f] = A \cap k[b]$). Since $\text{dim} \, A = 1$, the localization $A \otimes_{k[b^f]} k(b^f)$ must be integral over the field $k(b^f)$. Note that $A \otimes_{k[b^f]} k(b^f) \subset k(b)[a]$. The integral closure of the field $k(b^f)$ in the polynomial ring $k(b)[a]$ is just $k(b)$. Therefore we must have $A \subset k(b)$. This implies that $A = A \cap k[b] = k[b^f]$. It follows again that $X$ is spherical.
\end{proof}

We are left with the case of dimension 2. We start with a useful lemma.
\begin{lemma} \label{lemma: invariants of fraction field}
Let $X$ be an affine variety equipped with a left action of a unipotent group $U$. Then we have equality between $U$-invariants
$\text{Frac}(k[X]^U) = \text{Frac}(k[X])^{U}$.
\end{lemma}
\begin{proof}
See the argument in \cite{perri_intro_spherical}[4.1.16 (i)].
\end{proof}

For the following lemma, recall that the choice of a closed point $X$ in the open dense orbit allows us to view $X = Spec(R)$ for some left $SL_2$-stable subalgebra $R \subset k[SL_2]$.

\begin{lemma} \label{lemma: general normal affine when dim 2 is right T stable}
Let $X$ be an affine normal $SL_2$-variety with an open dense orbit. Assume that $X$ is not a homogeneous $SL_2$-variety. Suppose that the ring $k[X]^{\overline{U}}$ has dimension 2. Then $X$ admits a right $T$-action commuting with the $SL_2$-action.

More precisely, there exists a closed point $p \in X$ in the open dense orbit such that the corresponding $SL_2$-stable subalgebra $R \subset k[SL_2]$ is $T^{op}$-stable.
\end{lemma}
\begin{proof}
Let's start by choosing a closed point $p$ in the open dense orbit of $X$. Let $R$ be the $SL_2$-stable subalgebra of $k[SL_2]$ such that $X = Spec(R)$. Since $X$ is not a homogeneous $SL_2$-variety, there is some closed point in the complement of the open dense orbit. By the Hilbert-Mumford criterion \ref{lemma: hilbert mumford criterion}, we know that there exists a nonconstant one-parameter subgroup $\gamma$ such that $\text{lim}_{t \to 0} \, \gamma(t) \cdot p$ exists in $Spec(R)$. After right multiplying by and element of $SL_2(k)$, we can assume that $B_{\gamma} = \overline{B}$. 

Let $A = R^{\overline{U}} = R \cap k[a,b]$ be the corresponding admissible subalgebra of $k[a,b]$. We just need to show that $A$ is right $T$-stable. This is because the $T^{op}$-action commutes with the left $SL_2$-action. So if $A$ is $T^{op}$-stable, then its $SL_2$-span $R$ will also be $T^{op}$-stable. 

It suffices to prove that for every homogeneous polynomial $p \in A$, all of the $T^{op}$-weight components of $p$ are also in $A$.

We first claim that there exists some positive integer $f$ such that $b^f \in A$. Let $p \in A$ be a homogeneous polynomial of degree $n$. Up to a scalar multiple, we can write $p = a^m b^{n-m} +\sum_{i=0}^{m-1}z_ia^ib^{n-i}$ for some $m \leq n$. Since $B_{\gamma} = \overline{B}$, we conclude by Lemma \ref{lemma: criterion for spherical classes of limits} that $m \leq \left\lfloor \frac{n}{2} \right\rfloor$. If $m \neq \frac{n}{2}$, then the argument in Proposition \ref{prop: classification of spherical varieties cases} Case (C1) shows that the element $w^{p,p}_{2m}$ is of the form $zb^f$ for some $z \in k \setminus \{0\}$ and $f \geq 1$. So we would be done in this case. Suppose therefore that $m = \frac{n}{2}$. Consider the subalgebra $k[p] \subset A$. This inclusion can't be an equality because $A$ has dimension $2$ by assumption. This means that there is a nonzero homogeneous polynomial $h \in A \setminus k[p]$. We can write $h = a^sb^{t-s} + \sum_{i=0}^{s-1} z'_i a^i b^{t-i}$. If $s \neq \frac{t}{2}$, then by the same reasoning as above we can conclude that $b^f \in A$ for some $f \geq 1$. Otherwise, we can look at the nonzero polynomial $p^t - h^n$. Notice that this is a homogeneous polynomial of degree $nt$. The highest power of the variable $a$ must be strictly smaller than $\frac{nt}{2}$ because we are cancelling the monomials with the highest $a$-exponents in $p^t$ and $h^m$. Therefore, we can just apply the argument of \ref{prop: classification of spherical varieties cases}, Case (C1) to deduce that $b^f \in A$ for some $f \geq 1$.

Now let $f$ be the smallest positive integer such that $b^f \in A$. Let $\text{Stab}_p$ denote the stabilizer of the distinguished point $p$. Note that the algebra $R$ must be invariant under (i.e. pointwise fixed by) the right action of $\text{Stab}_p$. This in particular implies that $A$ is invariant under the right action of $\text{Stab}_p$. Since $b^f \in A$, this means that $b^f$ is fixed by the right action of $\text{Stab}_p$. It can be checked that the right stabilizer of $b^f$ is the subgroup $\mu_f \ltimes \overline{U}$ in $\overline{B} = T \ltimes \overline{U}$. Note that we must have $\text{dim} \, \text{Stab} = 0$. Otherwise $ \overline{U} \subset \text{Stab}_p$, and so we would have 
\[A \subset k[a,b]^{\text{Stab}_p^{op}} \subset k[a,b]^{\overline{U}^{op}} = k[b] \]
This is a contradiction to the assumption that $\text{dim}\, A = 2$. Therefore $\text{Stab}_p$ is a finite subgroup of $\mu_f \ltimes \overline{U}$. We can use an element of the connected group $\overline{B}$ to conjugate $\text{Stab}_p$ to a subgroup of $T$. So after right multiplying our algebra $R$ by an element of $\overline{B}$, we can assume without loss of generality that $\text{Stab}_p$ is a finite subgroup of $T$. We still know that $\text{Stab}_p$ fixes $b^f$, and therefore $\text{Stab}_p \subset \mu_f$.

We claim that in fact $\text{Stab}_p = \mu_f$. Say $\text{Stab}_p = \mu_e \subset \mu_f$ for some $e$ dividing $f$. Note that $Spec(R)$ contains $SL_2/\text{Stab}_p$ as an open subscheme. The variety $SL_2/\text{Stab}_p$ is $Spec$ of the algebra of right $\mu_e$-invariants $k[SL_2]^{\mu_e^{op}}$. Now we can use Lemma \ref{lemma: invariants of fraction field} in order to get the following chain of equalities.
\[ \text{Frac}(A) = \text{Frac}(R^{\overline{U}}) = \text{Frac}(R)^{\overline{U}} = \text{Frac}\left(k[SL_2]^{\mu_e^{op}}\right)^{\overline{U}} = \text{Frac}\left( (k[SL_2]^{\mu_e^{op}})^{\overline{U}}\right)\]
Since the left and right $SL_2$-actions commute, we can exchange the order of taking right and left invariants. This way we get
\begin{gather*}
 \text{Frac}(A) = \text{Frac}\left( (k[SL_2]^{\overline{U}})^{\mu_e^{op}}\right) = \text{Frac}\left(k[a,b]^{\mu_e^{op}} \right) = \text{Frac}(k[a^e, b^e, ab])
\end{gather*}
In particular $b^e \in \text{Frac}(A)$. Since $A$ is normal and $b^f \in A$, we must have $b^e \in A$. By minimality of $f$ we have $e=f$. This concludes the proof that $\text{Stab}_p = \mu_f$. Since $A$ is fixed by the right action of $\text{Stab}_p$, we must have $A \subset k[a,b]^{\mu_f^{op}} =  k[a^f,b^f, ab]$.

We are now ready to finish the proof of the proposition. Let $p$ be a homogeneous polynomial in $A$. We will distinguish two cases.

\begin{enumerate}[(C1)]
    \item $f$ is odd.
    Since $A \subset k[a^f, b^f,ab]$, we can write 
    \[p = a^{m} b^{n-m} +\sum_{i=1}^{\left\lfloor \frac{m}{f} \right\rfloor}z_i \,a^{m-if} \, b^{n-m+if}\]
    for some $z_i \in k$. The $T^{op}$-weight components of $p$ are all the monomial summands $z_i a^{m-if}b^{n-m+if}$ and $a^{m}b^{n-m}$ in the sum above. We claim that we actually have $a^{m-if} b^{n-m+if} \in A$ for all $0 \leq  i \leq \left\lfloor \frac{m}{f} \right\rfloor$. This would imply that all $T^{op}$-weight components are in $A$, thus concluding our proof. 

Let's prove this last claim. We will assume without loss of generality that $m \geq f$, otherwise the claim is obvious. Let us set $p_2 \vcentcolon = b^{2f}$. We know that $b^{2f} \in A$ by our work above. Let's look at the highest weight vector
\[ w^{p,p_2}_{f} = \sum_{\alpha =f}^{n+f} y_{\alpha,f}^{p,p_2} \, a^{\alpha-f} b^{n+f-\alpha} \]
Notice that $y_{\alpha,f}^{p,p_2} = 0$ for all $\alpha > m$. This is because in the sum for $y_{\alpha,f}^{p,p_2}$ there are no notrivial pairs $(i,j)$ with $i+j >m$. So we actually have
\[ w^{p,p_2}_{f} = \sum_{\alpha =f}^{m} y_{\alpha,f}^{p,p_2} \, a^{\alpha-f}b^{n+f-\alpha} \]
The coefficient $y_{m,f}^{p,p_2}$ can be checked to be
\[y_{m,f}^{p,p_2} = \sum_{e=0}^f (-1)^e \frac{\binom{f}{e}}{\binom{n}{e} \binom{2f}{f-e}} \binom{n-m}{e} \binom{2f}{f-e}\]
This can be simplified to
\[ y_{m,f}^{p,p_2} = \sum_{e=0}^f (-1)^e \frac{\binom{f}{e} \binom{n-m}{e}}{\binom{n}{e} }  \]
By Lemma \ref{lemma: third binomial sum lemma}, we have $y_{m,f}^{p,p_2} \neq 0$. Therefore, we can divide by $y_{m,f}^{p,p_2}$ and write
\[\frac{1}{y_{m,f}^{p,p_2}} w_f^{p,p_2} = a^{m-f}b^{n-m+f} + \sum_{i=2}^{\left\lfloor \frac{m}{f} \right\rfloor}\frac{y_{m-i+1,f}^{p,p_2}}{y_{m,f}^{p,p_2}} a^{m-if} b^{n-m+if}\]
Let us call $q_0 \vcentcolon = p$ and $q_{1} \vcentcolon = \frac{1}{y_{m,f}^{p,p_2}} w_f^{p,p_2}$. For every $0 \leq j \leq m$, we can define recursively $q_j \vcentcolon = \frac{1}{y_{m-j+1,f}^{q_{j-1}, \, p_2}} w^{q_{j-1}, \,p_2}$. The computation above shows that we can write
\[ q_j = a^{m-jf}b^{n-m+jf} + \sum_{i=j+1}^{\left\lfloor \frac{m}{f} \right\rfloor} z_{j,i} a^{m-if} b^{n-m+if}\]
for some coefficients $z_{j,i} \in k$. By assumption $A$ is admissible, so $q_j \in A$ for all $j$.
We will now show by descending induction that $a^{m-if}b^{n-m+if} \in A$ for all $0 \leq i \leq \left\lfloor \frac{m}{f} \right\rfloor$. For the base case we have $q_{\left\lfloor \frac{m}{f} \right\rfloor} = b^{n-m+ \left\lfloor \frac{m}{f} \right\rfloor f} \in A$. Let $0 \leq j \leq \left\lfloor \frac{m}{f} \right\rfloor$. Assume by induction that we know that $a^{m-if}b^{n-m +if} \in A$ for all $j+1\leq i \leq \left\lfloor \frac{m}{f} \right\rfloor$. Then, we can write the monomial $a^{m-jf} b^{n-m+jf}$ as
\[ a^{m-jf} b^{n-m+jf} = q_{j} - \sum_{i=j+1}^{\left\lfloor \frac{m}{f} \right\rfloor} z_{j,i} a^{m-if} b^{n-m+if}\]
Since all of the summands in the right-hand side are in $A$, we conclude that $a^{m-jf}b^{n-m+jf} \in A$. So we are done by induction.
    
    \item $f$ is even.
     Since $A \subset k[a^f, b^f,ab]$, we can write 
    \[p = a^{m} b^{n-m} +\sum_{i=1}^{\left\lfloor \frac{2m}{f} \right\rfloor}z_i \,a^{m-i\frac{f}{2}} \, b^{n-m+i\frac{f}{2}}\]
    for some $z_i \in k$. The $T^{op}$-weight components of $p$ are all the monomial summands $z_i a^{m-i\frac{f}{2}}b^{n-m+i\frac{f}{2}}$ and $a^{m}b^{n-m}$ in the sum above. We can do the same argument as in case $(C1)$ but replacing every instance of $f$ by $\frac{f}{2}$ (e.g. $p_2 =b^f$ and we look at the highest weight vector $w^{p,p_2}_{\frac{f}{2}}$).
    
\end{enumerate}
\end{proof}

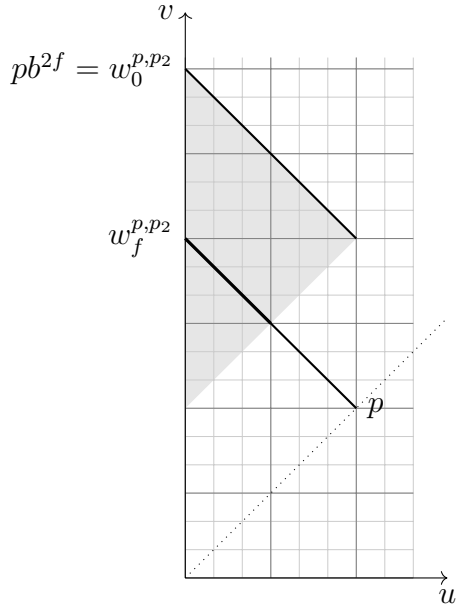
\begin{figure}[ht!]
\centering

\begin{tikzpicture}[scale=1.5]

\fill[gray!20] (0,1.5) -- (1.5,3) -- (0,4.5) -- (0,1.5);

\draw[step=.25cm,gray!40,very thin] (0,0) grid (2,4.6);
\draw[step=.75cm,gray,very thin] (0,0) grid (2,4.6);

\draw [->] (0,0) -- (2.3,0) node[anchor=north]{$u$};
\draw [->] (0,0) --  (0,5.0) node[anchor=east]{$v$};
\draw[dotted] (0,0) -- (2.3,2.3);

\draw [-,thick] (0,3) -- (1.5,1.5)    node[fill=none,anchor=west]{$p$};

\draw [-,thick] (1.5,3) -- (0,4.5)  node[fill=none,anchor=east]{$p b^{2f} = w_0^{p,p_2}$};
\draw [-,very thick] (0.75, 2.25) -- (0,3)  node[fill=none,anchor=east]{$w_f^{p,p_2}$};

\end{tikzpicture}

\caption{Application of the criterion in Proposition \ref{prop: criterion for admissible subalgebras} to $(p,b^{2f})$ in the proof of Lemma \ref{lemma: general normal affine when dim 2 is right T stable}, case (C1). The darker grid represents the sublattice induced by $f = 3$.}
\end{figure}

\begin{defn} \label{defn: Sq}
Let $q\geq 1$ be a rational number. We define $S_q$ to be the subalgebra of $k[a,b]$ generated by all monomials $a^i b^j$ satisfying $j \geq qi$. We will denote by $V_{S_q}$ the left $SL_2$-subrepresentation of $k[SL_2]$ spanned by $S_q$.
\end{defn} 

\begin{lemma} \label{lemma: Sq are admissible}
The algebra $S_q$ is admissible in the sense of Definition \ref{defn: admissible subalgebra}.
\end{lemma}
\begin{proof}
 $Spec(S_q)$ is a toric variety for the torus $(T \times T^{op})/\mu_2$, where we are embbeding $\mu_2$ diagonally. Motivated by this, we will use the notation for toric varieties as in 
\cite{cox_toric}.

Let $\Sigma_q$ be the semigroup in $\mathbb{N}^2 \subset \mathbb{Z}^2$ defined by
\[ \Sigma_q = \left\{ (i,j) \in \mathbb{N}^2 \; \mid \; a^ib^j \in S_q\right\} \]
Equivalently 
\[\Sigma_q = \left\{ (i,j) \; \mid \; j \geq qi \right\}\]
Notice that this semigroup is saturated in $\mathbb{Z}^2$. Actually it corresponds to the rational cone $\mathbb{R}^{\geq 0} (0,1) + \mathbb{R}^{\geq 0}(1, q)$ inside the vector space $\mathbb{R}^2 = \mathbb{Z}^2 \otimes \mathbb{R}$. By \cite{cox_toric}[Thm. 1.3.5] and \cite{cox_toric}[Prop. 1.2.17], the algebra $S_q = k[\Sigma_q]$ is normal and finitely generated.

In order to prove condition $(c)$ in the definition of admissible algebra, we can use Proposition \ref{prop: criterion for admissible subalgebras}. In this case it suffices to check the condition when $p_1, p_2$ are monomials in $a,b$. Let $p_1 = a^ib^j$ and $p_2 = a^lb^f$ be monomials in $S_q$. We need to look at
\[ w^{p_1,p_2}_{s} = \sum_{\alpha = s}^{i+j+l+f-s} y_{\alpha,s}^{p_1,p_2} a^{\alpha-s}b^{i+j+l+f-\alpha-s}\]
Notice that in this case $y_{\alpha,s} = 0$ unless $\alpha = i+l$, because there is only one nontrivial coefficient in both $p_1$ and $p_2$. So
\[ w^{p_1,p_2}_{s} = y_{i+l,s}^{p_1,p_2}\, a^{i+l-s} \,b^{j+f-s}\]
In order to show that this is in $S_q$, it suffices to show that $j+f-s \geq q \cdot(i+l-s)$. This follows from combining the inequalities $j \geq qi$, $f \geq ql$ and $-s \geq -qs$ (the last one because $q \geq 1$).
\end{proof}

\begin{defn} \label{defn: S_q^f}
For any positive integer $f$ and any rational number $q \geq 1$, we will denote by $S_q^{(f)}$ the subalgebra of $k[SL_2]$ defined by $S_q^{(f)} \vcentcolon = S_q \cap k[a^f, b^f,ab]$. We will denote by $R_{q}^{(f)}$ the left $SL_2$-subrepresentation of $k[SL_2]$ spanned by $S_q^{(f)}$.
\end{defn}

\begin{lemma} \label{lemma: the algebras S_q^f are admissible}
\begin{enumerate}[(a)]
    \item For every positive integer $f$ and rational number $q \geq 1$, the algebra $S_q^{(f)}$ is admissible. In particular $R^{(f)}_q$ is a subalgebra of $k[SL_2]$.
    \item Let $f_1,f_2$ be positive integers and let $q_1,q_2 \geq 1$ be rational numbers. If $f_1 \neq f_2$ or $q_1 \neq q_2$, then the varieties $Spec\left(R_{q_1}^{(f_1)}\right)$ and $Spec\left(R_{q_2}^{(f_2)}\right)$ are not $SL_2$-isomorphic.
\end{enumerate}
\end{lemma}
\begin{proof}
\begin{enumerate}[(a)]
    \item Notice that we have $S_q^{(f)} = \left(S_q\right)^{\mu_f^{op}}$ by definition. By Lemma \ref{lemma: Sq are admissible}, the algebra $S_q$ is finitely generated and normal. By \cite{inv_theory_neusel}, this implies that the algebra of invariants $\left(S_q\right)^{\mu_f^{op}}$ is finitely generated. Also $S_q$ being normal implies that $\left(S_{q}\right)^{\mu_f^{op}}$ is normal.

We are left to prove property $(c)$ in Definition \ref{defn: admissible subalgebra}. This amounts to showing that $R_{q}^{(f)}$ is an algebra. But notice that $R_{q}^{(f)} = \left(V_{S_q}\right)^{\mu_f^{op}}$. By Lemma \ref{lemma: Sq are admissible} we know that $V_{S_q}$ is closed under multiplication. We conclude that the $\mu_f^{op}$-invariants $\left(V_{S_q}\right)^{\mu_f^{op}}$ are also closed under multiplication, as desired.

\item It follows from the argument in the proof of Lemma \ref{lemma: general normal affine when dim 2 is right T stable} that the stabilizer of the open dense orbit in $Spec\left(R_{q_1}^{(f_1)}\right)$ (resp. $Spec\left(R_{q_2}^{(f_2)}\right)$) is $\mu_{f_1}$ (resp. $\mu_{f_2}$). If $f_1 \neq f_2$, this shows that $Spec\left(R_{q_1}^{(f_1)}\right)$ and $Spec\left(R_{q_2}^{(f_2)}\right)$ can't possibly be $SL_2$-isomorphic. Assume $f_1 = f_2 = f$. We can look at the algebras of $\overline{U}$-invariants
\[ \left(R_{q_1}^{(f)}\right)^{\overline{U}} = S_{q_1}^{(f)}\]
\[ \left(R_{q_2}^{(f)}\right)^{\overline{U}} = S_{q_2}^{(f)}\]
Suppose without loss of generality that $q_1 > q_2$. Then we have $S_{q_1}^{(f)} \subsetneq S_{q_2}^{(f)}$. This implies that $S_{q_1}^{(f)}$ and $S_{q_2}^{(f)}$ are not isomorphic even as left $T$-representations. This is because for $n>>0$, the dimension of the $T$-weight space $\left(S_{q_1}^{(f)}\right)_n$ must be strictly smaller than the dimension of the $T$-weight space $\left(S_{q_2}^{(f)}\right)_n$.

\end{enumerate}

\end{proof}

\begin{prop} \label{prop: general normal affine when dim 2}
Let $X$ be an affine normal $SL_2$-variety with an open dense orbit. Assume that $X$ is not a homogeneous $SL_2$-variety. Suppose that the ring $k[X]^{\overline{U}}$ has dimension 2. Then there exists a unique positive integer $f$ and a unique rational number $q \geq 1$ such that $X$ is $SL_2$-isomorphic to $Spec\left(R_q^{(f)}\right)$.
\end{prop}
\begin{proof}
By Lemma \ref{lemma: general normal affine when dim 2 is right T stable}, we can choose a closed point $p$ in the dense orbit such that
\begin{enumerate}[(a)]
    \item $\text{lim}_{t \to 0} \, \gamma(t) \cdot p$ exists in $X$ for some $\gamma$ with $B_{\gamma} = \overline{B}$.
    \item The stabilizer of $p$ is the $f$-torsion subgroup $\mu_f \subset T$.
\end{enumerate}
We can use $p$ to view $X$ as $Spec(R)$ for some $SL_2$-stable algebra $R \subset k[SL_2]$. Let $A = R \cap k[a,b]$ be the corresponding admissible subalgebra of $k[a,b]$. It suffices to show that $A = S_q^{(f)}$ for some $q \geq 0$. By Lemma \ref{lemma: general normal affine when dim 2 is right T stable}, $A$ is generated by $T\times T^{op}$-weights in $k[a,b]$. These are just the monomials in $k[a,b]$, so $A$ is determined by a semigroup on the set of monomials. Let $\Sigma_A$ denote the semigroup of $\mathbb{N}^2$ given by
\[ \Sigma_A \vcentcolon = \left\{ (i,j) \in \mathbb{N}^2 \; \mid \; a^ib^j \in A\right\} \]
Also, let $\Sigma^{(f)}$ be the semigroup of $\mathbb{N}^2$ determined by $k[a^f,b^f,ab]$
\[ \Sigma^{(f)} \vcentcolon = \left\{ (i,j) \in \mathbb{N}^2 \; \mid \; a^ib^j \in k[a^f,b^f,ab] \right\} = \left\{ (i,j) \in \mathbb{N}^2 \; \mid \; f \, \, | \, \, (j-i) \right\} \]
We have seen during the proof of Lemma \ref{lemma: general normal affine when dim 2 is right T stable} that there exists a positive integer $f$ such that $A \subset k[a^f, b^f, ab]$ and that
\[\text{Frac}(A) = k(a,b)^{\mu_f^{op}} = k(a^f, b^f,ab)\]
This means that $\Sigma_A \subset \Sigma^{(f)}$ and $\mathbb{Z}\Sigma_A = \mathbb{Z} \Sigma^{(f)}$. Since $A$ is normal, the proof of \cite{cox_toric}[Thm. 1.3.5] implies that $A$ is saturated in $\mathbb{Z}\Sigma^{(f)}$. Also, $\Sigma_A$ is generated by a finite set of monomial generators of $A$. Since $\Sigma_A$ is finitely generated and saturated in $\mathbb{Z}\Sigma^{(f)}$, we must have that $\Sigma_A = C \cap \Sigma^{(f)}$ for some rational cone $C$ on the vector space $\mathbb{R}^2 = \mathbb{Z}^2 \otimes \mathbb{R}$. In two dimensions, such a cone must be of the form $\mathbb{R}^{\geq 0}v_1 + \mathbb{R}^{\geq 0} v_2$ for two (possibly equal) lattice vectors $v_1, v_2 \in \mathbb{Z}^2$. Since $\mathbb{Z} \Sigma_{A} = \mathbb{Z}\Sigma^{(f)}$, this means that $v_1$ and $v_2$ must be linearly independent.

By definition, we have $\Sigma_A \subset \mathbb{N}^2$. Since $B_{\gamma} = \overline{B}$, Lemma \ref{lemma: criterion for spherical classes of limits} implies that $i\leq j$ for all $(i,j)\in \Sigma_A$. So the cone $C$ must be contained in the cone $\mathbb{R}^{\geq 0}(0,1) + \mathbb{R}^{\geq 0} (1,1)$. By the proof of Lemma \ref{lemma: general normal affine when dim 2 is right T stable}, we have $b^f \in A$. This means that $(0,1) \in C$, and so one of the vectors can be taken to be $v_1= (0,1)$. The other linearly independent vector is of the form $v_2 = (n_1,n_2)$ for some positive integers $n_1 \leq n_2$. So the cone is $C = \mathbb{R}^{\geq 0}(0,1)+ \mathbb{R}^2(n_1,n_2)$. We conclude that 
\[A = k[\Sigma_A] = k[C\cap \Sigma^{(f)}]= S_{\frac{n_2}{n_1}}^{(f)}\] 
Uniqueness follows from part $(b)$ of Lemma \ref{lemma: the algebras S_q^f are admissible}.
\end{proof}

Now we can put everything together to get a classification of all affine normal $SL_2$-varieties with an open dense orbit.
\begin{thm}\label{mainresult}
Let $X$ be an affine normal $SL_2$-variety with an open dense orbit. Then $X$ is $SL_2$-isomorphic to exactly one of the following
\begin{enumerate}[(1)]
    \item A homogenenous space $SL_2/H$, where $H$ is as in Proposition \ref{prop: classification of homogeneous $SL_2$-vars}.
    \item (See Definition \ref{defn: Rf}) The spherical variety $SL_2 /\!/ (\mu_f \ltimes \overline{U}) = Spec\left(R_{\infty}^{(f)}\right)$ for a unique positive integer $f$.
    \item (See Definition \ref{defn: S_q^f}) $Spec\left(R^{(f)}_q\right)$ for a unique positive integer $f$ and a unique rational number $q \geq 1$ . In this case the stabilizer of the dense orbit is $\mu_f$.
\end{enumerate}
\end{thm}
\begin{proof}
This follows by combining Proposition \ref{prop: classification of homogeneous $SL_2$-vars}, Proposition \ref{prop: general normal affine when dim 1} and Proposition \ref{prop: general normal affine when dim 2}.
\end{proof}

We end this section by describing explicitly the subalgebra $R_q^{(f)} \subset k[SL_2]$. Let $\{a^{i_m}b^{j_m}\}_{m \in I}$ be a finite set of monomial generators of $S_q^{(f)}$. This set can be obtained in practice by imitating the proof of Gordan's lemma \cite{cox_toric}[Prop. 1.2.17]. In other words, we look at $\Sigma^{(f)} \cap \Delta$, where $\Delta$ is a fundamental domain for the discrete subgroup of $\mathbb{R}^2$ generated by $(0,1)$ and $(n_1,n_2)$. Here $n_1,n_2$ are some positive integers such that $\frac{n_2}{n_1} = q$ and $n_2-n_1$ is divisible by $f$.
 
 For each $m \in I$, let us denote by $V^{(m)}$ the left $SL_2$-subrepresentation of $k[SL_2]$ spanned by the highest weight vector $a^{i_m} b^{j_m}$. A simple $SL_2$ computation shows that $V^{(m)}$ has a basis $\left(v_{s}^{(m)}\right)_{s=0}^{i_m +j_m}$, where
 \[v_{s}^{(m)} = \sum_{i+j =s} \binom{i_n}{i} \binom{j_m}{j} \,a^{i_m-i} \,c^{i} \,b^{j_m-j} \,d^j\]
By the proof of the finite generation statement in Lemma \ref{lemma: admissible subalgebras = pointed variaties with dense orbit}, the corresponding $SL_2$-stable subalgebra $R_q^{(f)}$ is generated as a $k$-algebra by the elements $v_{s}^{(m)}$ for $m \in I$ and $0 \leq s \leq i_m + j_m$.

\end{section}

\appendix
\begin{section}{Some lemmas about binomial sums}
In this appendix we collect some lemmas on binomial sums that we use for our classification results.

\begin{lemma} \label{lemma: first binomial sum lemma}
Let $i,j$ be nonnegative integers with $i \leq j$. Then, we have
\[ \sum_{e=0}^{2i} (-1)^e \frac{\binom{2i}{e}}{\binom{i+j}{e} \binom{i+j}{2i-e}} \binom{j}{e} \binom{j}{2i-e} = (-1)^i \, \frac{\binom{2i}{i}}{\binom{i+j}{i}^2} \,  \]

In particular this sum is never $0$.
\end{lemma}
\begin{proof}
We can use the identity
\[  \frac{\binom{j}{e} \binom{j}{2i-e} }{\binom{i+j}{e} \binom{i+j}{2i-e}} = \frac{1}{\binom{i+j}{i}^2} \binom{j+i-e}{i}\binom{j-i+e}{i}\]
to write the sum as
\[ \frac{1}{\binom{i+j}{i}^2} \cdot \sum_{e=0}^{2i} (-1)^e \binom{2i}{e} \binom{j+i-e}{i}\binom{j-i+e}{i}\]
We are reduced to showing the identity
\[ \sum_{e=0}^{2i} (-1)^e \binom{2i}{e} \binom{j+i-e}{i}\binom{j-i+e}{i} = (-1)^i\binom{2i}{i} \]
Let's first show that the left hand side does not depend on $j$. Let us fix $i$ and view the expression as a polynomial in $j$. We can expand
\[ \binom{j+i-e}{i}\binom{j-i+e}{i} = \sum_{s=0}^{2i}  p_{s}(e)\,  j^{2i-s}\]
where $p_{s}(e)$ is a polynomial on $e$ with coefficients in $\mathbb{Q}[i]$. Replacing this in the sum above, we get a polynomial in $j$ given by
\[ \sum_{s=0}^{2i} \left( \sum_{e=0}^{2i} (-1)^e\binom{2i}{e} \, p_{s}(e) \right) \, j^{2i-s}\]
In order to show that the sum does not depend on $j$, it suffices to show that $\sum_{e=0}^{2i} (-1)^e\binom{2i}{e}p_{s}(e) = 0$ for $s<2i$. Notice that $p_{s}(e)$ has degree at most $s$ in the variable $e$. So we can write
\[p_{s} = \sum_{f=0}^{s} q_{s,f}(i) \, e^{f}\]
for some $q_{s,f} \in \mathbb{Q}[i]$. We can use this to get
\[ \sum_{e=0}^{2i} (-1)^e\binom{2i}{e} \, p_{s}(e) = \sum_{f=0}^{s} q_{s,f}(i) \sum_{e=0}^{2i}(-1)^e \binom{2i}{e} e^{f} \]
Therefore we are reduced to showing that $\sum_{e=0}^{2i}(-1)^e \binom{2i}{e} e^{f} = 0$ for $f \leq s<2i$. In order to show this, we use the binomial theorem. Notice that the binomial theorem implies that
\[ \sum_{e=0}^{2i} x^e \binom{2i}{e} e^f = \left(x\frac{d}{dx}\right)^{f} (1+x)^{2i} \]
We can plug in $x=-1$ in the left hand side to get $\sum_{e=0}^{2i}(-1)^e \binom{2i}{e} e^{f}$. On the other hand, if we plug in $x=-1$ in the right hand side we get $0$ (one can see this by using the product rule and $f<2i$).

We conclude that the sum
\[ \sum_{e=0}^{2i} (-1)^e \binom{2i}{e} \binom{j+i-e}{i}\binom{j-i+e}{i} \]
does not depend on $j$. So we can evaluate by plugging in any value of $j$ we wish. If we plug in $j=i$, then we see that the only nonzero summand occurs when $e=i$. Hence we get
\begin{gather*}
    \sum_{e=0}^{2i} (-1)^e \binom{2i}{e} \binom{j+i-e}{i}\binom{j-i+e}{i} \xlongequal{j=i} \sum_{e=0}^{2i} (-1)^e \binom{2i}{e} \binom{2i-e}{i}\binom{e}{i} = (-1)^i\binom{2i}{i}
\end{gather*} 
as desired.
\end{proof}

\begin{lemma} \label{lemma: second binomial sum lemma}
Let $m$ be a positive integer and let $0 \leq i \leq m$. Then we have
\[ \sum_{e =0}^2 (-1)^e \frac{\binom{2}{e}}{\binom{2m}{e}\binom{2m}{2-e}} \binom{m}{e} \binom{m+i}{2-e} = \frac{mi^2-m^2}{2m^2\cdot(2m-1)} \]
\end{lemma}
\begin{proof}
This is just tedious algebra. One needs to sum three terms and operate with the fractions to arrive at the expression in the right-hand side.
\end{proof}

The following sum is needed for some of the arguments in the last section.
\begin{lemma} \label{lemma: third binomial sum lemma}
Let $f, m,n$ be positive integers with $f \leq m \leq n$. Then
\[\sum_{e=0}^{f} (-1)^e \frac{\binom{f}{e} \binom{n-m}{e}}{\binom{n}{e}} = \frac{\binom{m}{f}}{\binom{n}{f}}  \]

In particular, this sum is not $0$.
\end{lemma}
\begin{proof}
Fix the positive integer $f$. Let's interpret $n,m$ as variables. For psychological purposes, let's rename $n=x$ and $m=y$. The identity amounts to an equality of rational functions in the variables $x,y$
\[ \sum_{e=0}^f (-1)^e \binom{f}{e} \frac{\prod_{i=0}^{e-1}(x-y-i)}{\prod_{j=0}^{e-1}(y-i)} \; = \; \frac{\prod_{i=0}^{e-1}(y-i)}{\prod_{i=0}^{e-1}(x-i)} \]
We can rewrite this using the Pochhammer symbol to simplify notation
\[ \sum_{e=0}^f \frac{(-f)_e \, (y-x)_e}{n! (-x)_e} = \frac{(-y)_f}{(-x)_f}\]
The left-hand side is just the finite hypergeometric function $_2F_1(-f, \,y-x ; \,-x;\,1)$ (evaluated at $z=1$). The last equality of rational functions above is a direct consequence of the Chu-Vandermonde identity \cite{special_functions}[Cor. 2.2.3].
\end{proof}

\end{section}

\bibliography{classification_affine_sl2_vars.bib}

\begin{thebibliography}{MFK94}

\bibitem[AAR99]{special_functions}
George~E. Andrews, Richard Askey, and Ranjan Roy.
\newblock {\em Special functions}, volume~71 of {\em Encyclopedia of
  Mathematics and its Applications}.
\newblock Cambridge University Press, Cambridge, 1999.

\bibitem[BH08]{batyrev}
Victor Batyrev and Fatima Haddad.
\newblock On the geometry of {SL}(2)-equivariant flips.
\newblock {\em Moscow Mathematical Journal}, 8, 04 2008.

\bibitem[CLS11]{cox_toric}
David~A. Cox, John~B. Little, and Henry~K. Schenck.
\newblock {\em Toric varieties}, volume 124 of {\em Graduate Studies in
  Mathematics}.
\newblock American Mathematical Society, Providence, RI, 2011.

\bibitem[Dol]{dolgachev_mackay}
I.~Dolgachev.
\newblock Mc{K}ay correspondence.
\newblock \url{http://www.math.lsa.umich.edu/~idolga/McKaybook.pdf}.
\newblock Unpublished notes.

\bibitem[Kra85]{kraft}
H.~Kraft.
\newblock {\em Geometrische Methoden in der Invariantentheorie}.
\newblock Aspekte der Mathematik. Vieweg+Teubner Verlag, 1985.

\bibitem[LV83]{lunavust}
D.~Luna and Th. Vust.
\newblock Plongements d'espaces homogènes.
\newblock {\em Commentarii mathematici Helvetici}, 58:186--245, 1983.

\bibitem[MFK94]{mumford_git}
D.~Mumford, J.~Fogarty, and F.~Kirwan.
\newblock {\em Geometric invariant theory}, volume~34 of {\em Ergebnisse der
  Mathematik und ihrer Grenzgebiete (2) [Results in Mathematics and Related
  Areas (2)]}.
\newblock Springer-Verlag, Berlin, third edition, 1994.

\bibitem[NS02]{inv_theory_neusel}
Mara~D. Neusel and Larry Smith.
\newblock {\em Invariant theory of finite groups}, volume~94 of {\em
  Mathematical Surveys and Monographs}.
\newblock American Mathematical Society, Providence, RI, 2002.

\bibitem[Per]{perri_intro_spherical}
N.~Perrin.
\newblock Introduction to spherical varieties.
\newblock \url{http://relaunch.hcm.uni-bonn.de/fileadmin/perrin/spherical.pdf}.
\newblock Unpublished notes.

\bibitem[Pop73]{popov}
Vladimir Popov.
\newblock Quasihomogeneous affine algebraic varieties of the group {SL}(2).
\newblock {\em Izvestiya Mathematics}, 7:793–831, 10 1973.

\bibitem[Ric77]{richardson-affinehomegeneous}
R.~W. Richardson.
\newblock Affine coset spaces of reductive algebraic groups.
\newblock {\em Bull. London Math. Soc.}, 9(1):38--41, 1977.

\bibitem[Slo80]{slodowy_simple_sing}
Peter Slodowy.
\newblock {\em Simple singularities and simple algebraic groups}, volume 815 of
  {\em Lecture Notes in Mathematics}.
\newblock Springer, Berlin, 1980.

\bibitem[Spr77]{springer_invariant_thry}
T.~A. Springer.
\newblock {\em Invariant theory}.
\newblock Lecture Notes in Mathematics, Vol. 585. Springer-Verlag, Berlin-New
  York, 1977.

\end{thebibliography}
\bibliographystyle{alpha}

\footnotesize

  \textsc{Department of Mathematics, Cornell University,
    310 Malott Hall, Ithaca, New York 14853, USA}\par\nopagebreak
  \textit{E-mail address}, \texttt{ajf254@cornell.edu}
  
  \textit{E-mail address}, \texttt{rmh322@cornell.edu}
\end{document}